\documentclass{amsart}
\usepackage{amsmath}
\usepackage{amscd}
\usepackage{amssymb}
\usepackage{amsthm}
\usepackage{color}
\usepackage{tcolorbox}
\usepackage{fullpage}
\usepackage{longtable}
\usepackage{pdflscape}
\usepackage{afterpage}
\usepackage[numbers]{natbib}  
\usepackage[draft]{hyperref}

\theoremstyle{plain}
\newtheorem{prop}{Proposition}
\newtheorem{lemma}[prop]{Lemma}
\newtheorem{theorem}[prop]{Theorem}

\theoremstyle{definition}

\newtheorem{remark}[prop]{Remark}

\newenvironment{psmallmatrix}
  {\left(\begin{smallmatrix}}
  {\end{smallmatrix}\right)}
  
\newcommand{\cE}{\mathcal E}

\author{Haowu Wang}

\address{Max-Planck-Institut f\"{u}r Mathematik, Vivatsgasse 7, 53111 Bonn, Germany}

\email{haowu.wangmath@gmail.com}
\author{Brandon Williams}

\address{Lehrstuhl A für Mathematik, RWTH Aachen, 52056 Aachen, Germany}

\email{brandon.williams@matha.rwth-aachen.de}


\date{\today}
\begin{document}

\nocite{*}

\title{Graded rings of Hermitian modular forms with singularities}

\begin{abstract} We study graded rings of meromorphic Hermitian modular forms of degree two whose poles are supported on an arrangement of Heegner divisors. For the group $\mathrm{SU}_{2,2}(\mathcal{O}_K)$ where $K$ is the imaginary-quadratic number field of discriminant $-d$, $d \in \{4, 7,8,11,15,19,20,24\}$ we obtain a polynomial algebra without relations. In particular the Looijenga compactifications of the arrangement complements are weighted projective spaces.
\end{abstract}

\maketitle

\section{Introduction}

The ring of symmetric Hermitian modular forms of degree two over the number field $\mathbb{Q}(\sqrt{-3})$ was shown by Dern and Krieg \cite{DK1} to be a polynomial algebra without relations generated by forms of weights $4, 6, 9, 10, 12$. Their proof relies on the construction of modular forms with special divisors as Borcherds products, and has been applied to imaginary quadratic fields of other discriminants (\cite{DK2, bw2}). However, by \cite{W}, the algebra of symmetric Hermitian modular forms over the unitary group $\mathrm{U}_{2,2}(\mathcal{O}_K)$ or $\mathrm{SU}_{2,2}(\mathcal{O}_K)$ is freely generated if and only if the discriminant of the underlying number field is $-3$ or $-4$, and as the discriminant increases the ring structure rapidly becomes quite complicated.

In this paper we will instead consider rings $\mathcal{M}_*^!$ of meromorphic modular forms with poles supported on certain rational quadratic divisors. Looijenga \cite{L} found conditions that guarantee that every nonzero form in $\mathcal{M}_*^!$ has nonnegative weight and that $\mathcal{M}_*^!$ itself is finitely generated. The proj of $\mathcal{M}_*^!$ is then the Looijenga compactification of the complement of these rational quadratic divisors in the Hermitian modular fourfold, with properties similar to the Baily--Borel compactification. 

Among rings of the form $\mathcal{M}_*^!$ there are a surprising number of examples of free algebras of modular forms:

\begin{theorem} For each $d \in \{4,7,8,11,15,19,20,24\}$ there is a Heegner divisor $H_d$ for which the ring of symmetric meromorphic modular forms for the group $\mathrm{SU}_{2,2}(\mathcal{O}_K)$, $K = \mathbb{Q}(\sqrt{-d})$ is freely generated. In particular, the Looijenga compactification of the arrangement complement $\Gamma_K \backslash (\mathbf{H}_2 - H_d)$ is a complex weighted projective space of dimension four. 
\end{theorem}
Here $\Gamma_K$ is the group generated by $\mathrm{SU}_{2,2}(\mathcal{O}_K)$ and a certain reflection $\sigma$, such that the modular forms on $\Gamma_K$ are precisely the symmetric modular forms on $\mathrm{SU}_{2,2}(\mathcal{O}_K)$.
In Table \ref{Maintab1}, $D_{\ell}$ denotes an irreducible Heegner divisor of discriminant $\ell$. Note that for $d = 4$ we must take the group $\mathrm{SU}_{2,2}(\mathcal{O}_K)$, rather than the usual modular group $\mathrm{U}_{2,2}(\mathcal{O}_K)$.
For discriminant $d \in \{4, 7, 8, 11\}$ the modular forms are allowed to have poles precisely on the Siegel upper half-space $\mathbb{H}_2$ (viewed as the subset of symmetric matrices in the Hermitian upper half-space $\mathbf{H}_2$) and its conjugates under the modular group. We prove in Theorem \ref{th:classify} that these are the only such examples. We do not have a classification of all free algebras of meromorphic Hermitian modular forms, but from some searching it seems likely that there are none besides those mentioned above.
\begin{table}[htbp]
\caption{Free algebras of meromorphic modular forms.}\label{Maintab1}
\renewcommand\arraystretch{1.1}
\noindent\[
\begin{array}{ccc}
d & H_d & \text{Generator weights} \\
\hline
4 & D_{1/4} & 2,4,6,8,10 \\
7 & D_{1/7} & 2,3,4,7,8 \\
8 & D_{1/8} & 2,3,4,6,8 \\
11 & D_{1/11} & 2,3,4,5,6 \\
15 & D_{1/15} + D_{1/15} & 2,3,3,4,4 \\
19 & D_{1/19} + D_{4/19} & 1,2,3,4,5 \\
20 & D_{1/20} + D_{1/20} + D_{1/5} & 1,2,3,3,5 \\
24 & D_{1/24} + D_{1/24} + D_{1/6} & 1,2,3,3,4 \\
\hline
\end{array}\]
\end{table}

Hermitian modular forms of degree two also have the geometric interpretation as forms on moduli spaces of abelian fourfolds with CM, or of lattice-polarized K3 surfaces of Picard number 16, and the special divisors $D_{\ell}$ parameterize varieties with additional automorphisms. The theorem of Dern--Krieg above can be interpreted as a statement on K3 surfaces polarized by the root lattice $U \oplus E_8 \oplus E_6$, and the Jacobian of their generators is precisely the discriminant. As discussed in \cite{L}, some interesting moduli spaces can be realized as the complements of Heegner divisors in orthogonal modular varieties and the GIT compactifications of these moduli spaces are usually isomorphic to the Proj of the graded ring of meromorphic modular forms with constrained poles. It is natural to guess that the graded rings found here also have moduli space interpretations of this sort.

\section{Background}

\subsection{Lattices and modular forms}

Let $L = (L, Q)$ be an even integral lattice of signature $(n, 2)$, $n \ge 1$, where $Q : L \rightarrow \mathbb{Z}$ is its quadratic form and $\langle x, y \rangle = Q(x+y) - Q(x) - Q(y)$ its bilinear form. Fix one of the two connected components $\mathcal{D}(L)$ of $$\{[\mathcal{Z}] \in \mathbb{P}(L\otimes \mathbb{C}): \; \langle \mathcal{Z}, \mathcal{Z} \rangle =0, \; \langle \mathcal{Z}, \overline{\mathcal{Z}} \rangle < 0\}$$ and define $\mathcal{A}(L) = \{\mathcal{Z} \in L \otimes \mathbb{C}: \; [\mathcal{Z}] \in \mathcal{D}(L)\}.$
Let $\mathrm{O}(L)$ be the orthogonal group of $(L, Q)$. The \emph{full modular group} associated to $L$ is $$\mathrm{O}^+(L) = \{\gamma \in \mathrm{O}(L): \; \gamma(\mathcal{D}(L)) = \mathcal{D}(L)\}.$$ For a finite-index subgroup $\Gamma \le \mathrm{O}^+(L)$, a \emph{modular form} of weight $k \in \mathbb{Z}$ and character $\chi : \Gamma \rightarrow \mathbb{C}^{\times}$ is a holomorphic function $f : \mathcal{A}(L) \rightarrow \mathbb{C}$ satisfying $$f(t \mathcal{Z}) = t^{-k} f(\mathcal{Z}), \; \; t \in \mathbb{C}^{\times}$$ and $$f(\gamma \mathcal{Z}) = \chi(\gamma) f(\mathcal{Z}), \; \; \gamma \in \Gamma$$ as well as (for $n \le 2$) a boundedness condition ``at cusps". A typical choice for $\Gamma$ is the discriminant kernel $$\widetilde{\mathrm{O}}^+(L) = \{\gamma \in \mathrm{O}^+(L): \; \gamma x - x \in L \; \text{for all} \; x \in L'\},$$ where $$L' = \{x \in L \otimes \mathbb{Q}: \; \langle x, y \rangle \in \mathbb{Z} \; \text{for all} \; y \in L\}$$ is the dual lattice. A \emph{meromorphic} modular form is a meromorphic function $f$ satisfying the functional equations above as well as a meromorphy condition at cusps (which again is automatic for $n > 2$).

For any vector $\lambda \in L'$ of positive norm, define the \emph{rational quadratic divisor} $$\mathcal{D}_{\lambda}(L) = \{[\mathcal{Z}] \in \mathbb{P}(L \otimes \mathbb{C}): \; \langle \mathcal{Z}, \lambda \rangle = 0\}.$$
Let $m$ be a positive rational number. 
The union $$H(m, \gamma) = \bigcup_{\substack{\lambda \in L + \gamma \\ \lambda \; \text{primitive in} \; L' \\ Q(\lambda) = m}} \mathcal{D}_{\lambda}(L)$$ is locally finite and $\widetilde{\mathrm{O}}^+(L)$-invariant and therefore descends to an analytic divisor on $\tilde Y_L := \widetilde{\mathrm{O}}^+(L) \backslash \mathcal{D}(L)$, called a \emph{Heegner divisor} of discriminant $m$. We additionally define $$H(m) = \bigcup_{\gamma \in L'/L} H(m, \gamma)$$ and $$D(m) = \bigcup_{d\in \mathbb{N}} H(m / d^2),$$ such that $D(m)$ is the union of all $\mathcal{D}_{\lambda}(L)$ with $\lambda \in L'$ (not necessarily primitive) of norm $m$. Both $H(m)$ and $D(m)$ are $\mathrm{O}^+(L)$-invariant.

Modular forms on $\widetilde{\mathrm{O}}^+(L)$ can be constructed by lifting modular forms on congruence subgroups of $\mathrm{SL}_2(\mathbb{Z})$. We follow Borcherds \cite{B} and consider the input forms into this lift as vector-valued modular forms whose multiplier is the Weil representation attached to $L$. When $L$ has even rank this is the representation $$\rho : \mathrm{SL}_2(\mathbb{Z}) \longrightarrow \mathbb{C}[L' / L] = \mathrm{span}(e_x: \; x \in L' / L)$$ defined on $S = \begin{psmallmatrix} 0 & -1 \\ 1 & 0 \end{psmallmatrix}$ and $T = \begin{psmallmatrix} 1 & 1 \\ 0 & 1 \end{psmallmatrix}$ by $$\rho(S) e_x = \frac{i^{\mathrm{sig}(L) / 2}}{\sqrt{|L'/L|}} \sum_{y \in L'/L} e^{2\pi i \langle x,y \rangle} e_y$$ and $$\rho(T) e_x = e^{-2\pi i Q(x)} e_x.$$ (This is the dual of the representation $\rho_L$ of \cite{B} because in our convention $L$ has signature $(n, 2)$.) A \emph{nearly-holomorphic vector-valued modular form} of weight $k$ is a holomorphic function $F : \mathbb{H} \rightarrow \mathbb{C}[L'/L]$ that satisfies $$F(M \cdot \tau) = (c \tau + d)^k \rho(M) F(\tau)$$ and whose Fourier expansion about $\infty$ involves only finitely many negative exponents.

Borcherds \cite{B} defined two singular theta lifts that construct modular forms with respect to $\widetilde{\mathrm{O}}^+(L)$. Let $k \in \mathbb{N}_0$ and let $$F(\tau) = \sum_{x \in L'/L} \sum_{m \in \mathbb{Z} - Q(x)} c(m, x) q^m e_x, \; \; q = e^{2\pi i \tau}$$ be a nearly-holomorphic modular form of weight $1 + k - n/2$ whose Fourier coefficients are integers.

(1) If $k = 0$, there is a \emph{Borcherds product} $\Psi_F$, which is a meromorphic modular form with a character (or multiplier system) of weight $c(0, 0) / 2$ and divisor $$\mathrm{div} \, \Psi_F = \sum_{\substack{\lambda \in L' \\ Q(\lambda) > 0}} c(-Q(\lambda), \lambda) \mathcal{D}_{\lambda}(L).$$
(Note that the sum is not taken over primitive vectors. In particular $\Psi_F$ may be holomorphic even if some coefficients $c(-Q(\lambda), \lambda)$ are negative.)

(2) If $k \ge 1$, Borcherds defines a \emph{singular additive lift} $\Phi_F$, which is a meromorphic modular form of weight $k$ with trivial character on $\widetilde{\mathrm{O}}^+(L)$, all of whose poles have order $k$ and lie on rational quadratic divisors $\mathcal{D}_{\lambda}(L)$ with $c(-Q(\lambda), \lambda) \ne 0$. When $F$ is holomorphic, this coincides with the Gritsenko lift; in particular, $\Phi_F$ is also holomorphic, and if $F$ is a cusp form then $\Phi_F$ is also a cusp form unless $k = 1$.

Nearly-holomorphic input forms $F$ can be computed effectively \cite{bw3}. Most of the Borcherds products we will need were already tabulated in Appendix B of \cite{bw3}.

If $F$ is a modular form on the orthogonal group of a signature $(n, 2)$ lattice $L$ then its restriction, or pullback, to any rational quadratic divisor is a modular form of the same weight. There is an important generalization of the restriction map called the \emph{quasi-pullback}. For any (holomorphic) modular form $F(\mathfrak{Z})$ of weight $k$, with a zero of order $r \in \mathbb{N}_0$ on $\mathcal{D}_{\lambda}(L)$, we write $\mathcal{Z} = \mathfrak{z} + \lambda w$ with $\mathfrak{z} \subseteq \lambda^{\perp}$ and $w \in \mathbb{C}$, and define $$\mathrm{Q}F(\mathfrak{z}) := \lim_{w \rightarrow 0} \frac{F(\mathfrak{z} + \lambda w)}{w^r}.$$ This defines a modular form on $\mathcal{D}_{\lambda}(L)$ of weight $k + r$ which is a cusp form if $r > 0$.

The restriction map preserves the space of (singular) additive lifts. Slightly more precisely, for any form $F$ of weight $1 + k - n/2$, $k \ge 1$ and any $\lambda \in L$ of norm $m > 0$, we have the identity $$\Phi_F \Big|_{\mathcal{D}_{\lambda}(L)} = \Phi_{\vartheta F},$$ where $\vartheta F \in M_{1 + k - n/2 + 1/2}^!(\rho_{\lambda^{\perp}})$ is the \emph{theta-contraction} of $F$, obtained by multiplying the components by unary theta series of the form $\theta_a(\tau) = \sum_{n \in \mathbb{Z}} q^{(n + a)^2}$, $a \in \frac{1}{m}\mathbb{Z}$ and summing up. Ma \cite{M} showed under the assumption of Koecher's principle that the quasi-pullback of Borcherds products satisfies the same formula, $$\mathrm{Q} \Psi_F = \Psi_{\vartheta F},$$ by showing that both sides define a modular form with the same divisor. This identity was proved in a different way by Zemel \cite{Z} who showed that it holds even without the assumption of Koecher's principle.

\subsection{Hermitian modular forms of degree two}

Fix an imaginary-quadratic number field $K$ of discriminant $d_K$, with ring of integers $\mathcal{O}_K$ and dual lattice $$\mathcal{O}_K^{\#} = \{x \in K: \; \mathrm{tr}_{K/\mathbb{Q}}(xy) \in \mathbb{Z} \; \text{for all} \; y \in \mathcal{O}_K\}.$$

Let $\mathbf{H}_2$ be the Hermitian upper half-space of degree two: $$\mathbf{H}_2 = \{z=x + iy: \; x, y \in \mathbb{C}^{2 \times 2}, \; \overline{x} = x^T, \; \overline{y} = y^T, \; y \; \text{positive definite}\}.$$ This is acted upon by the split-unitary group $$\mathrm{U}_{2, 2}(\mathbb{C}) = \{M \in \mathrm{GL}_4(\mathbb{C}): \; M^T J \overline{M} = J\}, \; \; J = \begin{psmallmatrix} 0 & 0 & -1 & 0 \\ 0 & 0 & 0 & -1 \\ 1 & 0 & 0 & 0 \\ 0 & 1 & 0 & 0 \end{psmallmatrix}$$ by M\"obius transformations: $\begin{psmallmatrix} a & b \\ c & d \end{psmallmatrix} \cdot z = (az+b)(cz+d)^{-1}$ where $a,b,c,d$ are $(2 \times 2)$ blocks.

Let $$\Gamma \le \mathrm{U}_{2, 2}(\mathcal{O}_K)$$ be a finite-index subgroup. We denote by $\mathcal{A}_k(\Gamma)$ the space of automorphic forms of weight $k$, meaning meromorphic functions $f : \mathbf{H}_2 \rightarrow \mathbb{C}$ that satisfy $$f(\gamma z) = \mathrm{det}(cz+d)^k f(z), \; \; \gamma = \begin{psmallmatrix} a & b \\ c & d \end{psmallmatrix} \in \Gamma.$$ Any automorphic form extends to a meromorphic section of a vector bundle over $\Gamma \backslash \mathbf{H}_2$; this is a form of Koecher's principle. If $f$ is holomorphic then it has a Fourier expansion $$f(z) = \sum_{B \in \Lambda_K} c_f(B) e^{2\pi i \mathrm{tr}(Bz)}$$ where $$\Lambda_K = \Big\{ \text{hermitian  matrices} \; B = (b_{ij}), \; b_{ij} \in \mathcal{O}_K^{\#}\Big\},$$ and where $c_f(B)$ may be nonzero only if $B$ is positive semidefinite. The function $f$ is a cusp form if its nonzero coefficients $c_f(B)$ only appear when $B$ is positive definite.

We further define 
$$
\mathrm{SU}_{2,2}(\mathcal{O}_K):= \mathrm{U}_{2,2}(\mathcal{O}_K)\cap \mathrm{SL}_4(\mathbb{C})
$$
and remark that $\mathrm{SU}_{2,2}(\mathcal{O}_K)=\mathrm{U}_{2,2}(\mathcal{O}_K)$ if and only if $d_K \neq -3, -4$.

Hermitian modular forms of degree two are essentially the same as modular forms on $\mathrm{O}(4, 2)$.  Indeed, there is an isomorphism between $\mathrm{SU}_{2,2}(\mathcal{O}_K)$ and the subgroup $$\widetilde{\mathrm{SO}}^+(L) = \mathrm{ker}(\mathrm{det} : \widetilde{\mathrm{O}}^+(L) \rightarrow \{\pm 1\})$$ for the lattice $$L = U \oplus U \oplus \mathcal{O}_K,$$ where $U$ is $\mathbb{Z}^2$ with quadratic form $(x,y) \mapsto xy$ and where $\mathcal{O}_K$ is the lattice $\mathcal{O}_K$ together with its norm form $N_{K/\mathbb{Q}}$, and this leads to an identification between modular forms for these groups which is worked out in detail in \cite{Dern, HK}. The full discriminant kernel $\widetilde{\mathrm{O}}^+(L)$ is generated by $\widetilde{\mathrm{SO}}^+(L)$ and by the reflection $$\rho : U \oplus U \oplus \mathcal{O}_K \longrightarrow U \oplus U \oplus \mathcal{O}_K, \; \; (x_1,y_1,x_2,y_2,\beta) \mapsto (x_1,y_1,x_2,y_2,-\overline{\beta}),$$ whose action on $\mathbf{H}_2$ is the involution $\begin{psmallmatrix} \tau & z_1 \\ z_2 & w \end{psmallmatrix} \mapsto \begin{psmallmatrix} \tau & -z_2 \\ -z_1 & w \end{psmallmatrix}$ with automorphy factor $(+1)$. $\widetilde{\mathrm{O}}^+(L)$ also contains the reflection $$\sigma : U \oplus U \oplus \mathcal{O}_K \longrightarrow U \oplus U \oplus \mathcal{O}_K, \; \; (x_1,y_1,x_2,y_2,\beta) \mapsto (x_1,y_1,y_2,x_2,\beta)$$ whose action on $\mathbf{H}_2$ is the involution $\begin{psmallmatrix} \tau & z_1 \\ z_2 & w \end{psmallmatrix} \mapsto \begin{psmallmatrix} w & z_1 \\ z_2 & \tau \end{psmallmatrix}$ with automorphy factor $(+1)$. Finally, we remark that $\widetilde{\mathrm{SO}}^+(L)$ contains the map $$\iota : U \oplus U \oplus \mathcal{O}_K \longrightarrow U \oplus U \oplus \mathcal{O}_K, \; \; (x_1,y_1,x_2,y_2,\beta) \mapsto (-x_1,-y_1,-x_2,-y_2,\beta)$$ whose action on $\mathbf{H}_2$ is $\begin{psmallmatrix} \tau & z_1 \\ z_2 & w \end{psmallmatrix} \mapsto \begin{psmallmatrix} \tau & -z_1 \\ -z_2 & w \end{psmallmatrix}$ with automorphy factor $(-1)$.

Considering the transformations under $\sigma$ and $-\rho \iota$ shows that any Hermitian modular form $F$ of weight $k$ without character that arises from a modular form on $\widetilde{\mathrm{O}}^+(L)$ (including Maass lifts and Borcherds products) satisfies $$F \Big( \begin{pmatrix} \tau & z_1 \\ z_2 & w \end{pmatrix} \Big) = \varepsilon F\Big( \begin{pmatrix} w & z_1 \\ z_2 & \tau \end{pmatrix} \Big) = (-1)^k \varepsilon F\Big( \begin{pmatrix} \tau & z_2 \\ z_1 & w \end{pmatrix} \Big),$$ where $\varepsilon = 1$ if $F$ has trivial character and $\varepsilon = -1$ if $F$ has the determinant character. A Hermitian modular form is called \textit{symmetric} (resp. \textit{skew-symmetric}) if it is invariant (resp. anti-invariant) under the involution $\begin{psmallmatrix} \tau & z_1 \\ z_2 & w \end{psmallmatrix} \mapsto \begin{psmallmatrix} w & z_1 \\ z_2 & \tau \end{psmallmatrix}$. Symmetric Hermitian modular forms of weight $k$ and trivial character for $\mathrm{SU}_{2,2}(\mathcal{O}_K)$ can therefore be identified with modular forms of weight $k$ and trivial character for $\widetilde{\mathrm{O}}^+(L)$. 

Let us also mention here that under the local isomorphism from $\mathrm{O}(4, 2)$ to $\mathrm{U}(2, 2)$ the rational-quadratic divisors take the explicit form $$\Delta_{a,B,c} := \Big\{ z \in \mathbf{H}_2: \; a \cdot  \mathrm{det}(z) + \mathrm{tr}(Bz) + c = 0 \Big\}, \; \; a,c \in \mathbb{Z}, \; B \in \Lambda_K,$$ and the discriminant of this divisor (the norm of the corresponding $\lambda \in L'$) is $$\mathrm{disc}(\Delta_{a,B,c}) = ac - \mathrm{det}(B).$$ The Siegel upper half-space $\mathbb{H}_2$ always appears as the discriminant $1/|d_K|$ divisor $\Delta_{0, B, 0}$ associated to $B = \begin{psmallmatrix} 0 & i / \sqrt{|d_K|} \\ -i/\sqrt{|d_K|} & 0 \end{psmallmatrix}$; more generally, all other $\Delta_{a;B;c}$ can be mapped biholomorphically to $\mathbb{H}_2$ under the action of $\mathrm{U}_{2,2}(\mathbb{C})$. One other important example of a rational-quadratic divisor is $$\Delta_{0, \begin{psmallmatrix} 1 & 0 \\ 0 & -1 \end{psmallmatrix}, 0} = \Big\{\begin{pmatrix} \tau & z_1 \\ z_2 & w \end{pmatrix} \in \mathbf{H}_2: \; \tau = w \Big\},$$ which represents the Heegner divisor $H(1)$. Note that $H(1)=H(1,0)$ is irreducible because the lattice $L$ is maximal. In the language of the orthogonal group $\Delta_{0, \begin{psmallmatrix} 1 & 0 \\ 0 & -1 \end{psmallmatrix}, 0}$ is the mirror of $\sigma$.

For any even integer $k \ge 4$, the \emph{Hermitian Eisenstein series} $\cE_k$ may be defined as the theta lift (i.e. the Borcherds additive lift) of the vector-valued Eisenstein series $$E_{k, 0} = \sum_{M \in \Gamma_{\infty} \backslash \mathrm{SL}_2(\mathbb{Z})} e_0 |_k M,$$ or rather the Hermitian modular form corresponding to this modular form on $\mathrm{O}^+(L)$; cf. section 3.3 of \cite{Dern} for a formula for its Fourier series.

\subsection{Free algebras of meromorphic modular forms}

Let $L = (L, Q)$ be a lattice of signature $(n, 2)$ with locally symmetric space $\mathcal{D}(L)$ defined as in section 2.1. A \emph{hyperplane arrangement} (in the sense of Looijenga \cite{L}) will mean a finite family of Heegner divisors $H(n_i, \gamma_i)$, $i=1,...,N$ with the following property: \emph{for any one-dimensional intersection $\ell$ of hyperplanes $\lambda^{\perp}$ with $\lambda$ primitive and $\lambda \in L+\gamma_i$, $Q(\lambda) = n_i$ for some $i$, the one-dimensional lattice $\ell \cap L$ is positive-definite.}

For such a hyperplane arrangement $H$ let $\mathcal{M}_k^!$ denote the space of meromorphic orthogonal modular forms of weight $k$ for $\widetilde{\mathrm{O}}^+(L)$ which are holomorphic away from $H$. Looijenga proved (\cite{L}, Corollary 7.5) that $\mathcal{M}_k^! = \{0\}$ for all $k < 0$ and that $\mathcal{M}_0^! = \mathbb{C}$, and further that the algebra $$\mathcal{M}_*^! = \bigoplus_{k \ge 0} \mathcal{M}_k^!$$ is finitely generated. $\mathrm{Proj}\, \mathcal{M}_*^!$ has similar properties to the Baily--Borel compactification of $\tilde Y_L$ and is called the \emph{Looijenga compactification} of $\tilde Y_L$ associated to the arrangement $H$.

\begin{remark} In the examples corresponding to Hermitian modular forms throughout the rest of the paper, the hyperplane arrangements $H$ satisfy the stronger property that any nontrivial intersection of two hyperplanes in $H$ is already disjoint from $\mathcal{D}(L)$. In this case, the quasi-pullback of a modular form in $\mathcal{M}_k^!$ to any hyperplane $\lambda^{\perp} \in H$ is a holomorphic modular form of weight $k-m_\lambda$, where $m_\lambda$ is the multiplicity of the pole $\lambda^{\perp}$. By applying Koecher's principle to modular forms on any hyperplane, we see that $k\geq m_\lambda$. This proves the fact that $\mathcal{M}_*^!$ is generated by modular forms of positive weight in a more elementary way.  Here it is essential that the signature of the lattice is $(n, 2)$ with $n \ge 4$. 
\end{remark}

We will be interested in examples where $\mathcal{M}_*^!$ is a polynomial ring without relations. The results of \cite{W} show that whether a given set of modular forms generates a given graded ring of holomorphic modular forms can be read off of their Jacobian, and it is natural to guess that similar results apply to $\mathcal{M}_*^!$. We will show that they do. If $f_0,...,f_n : \mathcal{A}(L) \rightarrow \mathbb{C}$ are meromorphic modular forms of level $\Gamma \le \mathrm{O}^+(L)$ of weights $k_0,...,k_n$ and characters $\chi_0,...,\chi_n$, and $z_0,...,z_n$ are coordinates on $\mathcal{A}(L)$, then the Jacobian is $$J = J(f_0,...,f_n) = \mathrm{det} \Big( \partial f_i / \partial z_j \Big)_{i,j=0}^n.$$ After applying the chain rule we see that $J$ transforms as a modular form of weight $n + \sum_{i=0}^n k_i$ and character $$\chi = \mathrm{det} \otimes \bigotimes_{i=0}^n \chi_i.$$

The Jacobian satisfies the product rule in each component and every meromorphic modular form can be written as a quotient $f = g/h$ where $g, h$ are holomorphic. After applying the equation $$J(g/h, f_1,...,f_n) = \frac{1}{h^2} \Big( h \cdot J(g, f_1,...,f_n) - g \cdot J(h, f_1,...,f_n) \Big)$$ and the analogous equations in the other components we see that $J(f_0,...,f_n)$ is meromorphic with poles at worst where any of $f_0,...,f_n$ has a pole.

In the case of Hermitian modular forms the Jacobian becomes the Rankin--Cohen--Ibukiyama bracket: $$J(f_0,...,f_4)\Big( \begin{pmatrix} \tau & z_1 \\ z_2 & w \end{pmatrix} \Big) = \mathrm{det} \begin{pmatrix} k_0 f_0 & k_1 f_1 & k_2 f_2 & k_3 f_3 & k_4 f_4 \\ \partial_{\tau} f_0 & \partial_{\tau} f_1 & \partial_{\tau} f_2 & \partial_{\tau} f_3 & \partial_{\tau} f_4 \\ \partial_{z_1} f_0 & \partial_{z_1} f_1 & \partial_{z_1} f_2 & \partial_{z_1} f_3 & \partial_{z_1} f_4 \\ \partial_{z_2} f_0 & \partial_{z_2} f_1 & \partial_{z_2} f_2 & \partial_{z_2} f_3 & \partial_{z_2} f_4 \\ \partial_{w} f_0 & \partial_{w} f_1 & \partial_{w} f_2 & \partial_{w} f_3 & \partial_{w} f_4 \end{pmatrix}.$$

Throughout the paper, we denote by $\mathcal{M}_*^!$ the ring of symmetric meromorphic Hermitian modular forms on $\mathrm{SU}_{2,2}(\mathcal{O}_K)$ with poles along the hyperplane arrangement $H$, that is, the ring of meromorphic modular forms for $\widetilde{\mathrm{O}}^+(L)$ which are holomorphic on $\mathcal{D}(L)\setminus H$. 
\begin{theorem}\label{th:j} Let $f_0,...,f_4 \in \mathcal{M}_*^!$. Suppose the Jacobian $J = J(f_0,...,f_4)$ is nonzero. 
\begin{itemize}
\item[(i)] The Jacobian $J$ vanishes on the Heegner divisor $D(1)\setminus H$.
\item[(ii)] Suppose $J$ has only a simple zero on $D(1)\setminus H$, and that all other poles or zeros of $J$ are contained in the hyperplane arrangement $H$. Then $$\mathcal{M}_*^! = \mathbb{C}[f_0,...,f_4].$$
\end{itemize}
\end{theorem}
By \cite[Theorem 2.5 (2)]{W}, the Jacobian $J$ is nonzero if and only if $f_0,...,f_4$ are algebraically independent.
\begin{proof} (i) Since the Jacobian has the determinant character for $\widetilde{\mathrm{O}}^+(L)$, it vanishes on all mirrors of reflections in the discriminant kernel which are not contained in $H$ (see \cite[Theorem 2.5 (4)]{W}). Recall that the mirrors of reflections in $\widetilde{\mathrm{O}}^+(L)$ are exactly the rational quadratic divisors associated to vectors of norm $1$ in $L$. For the lattices corresponding to Hermitian modular forms, the union of these mirrors is the Heegner divisor $D(1)$ (in fact, $D(1)=H(1)+H(1/4)$ if $4|d_K$, otherwise $D(1)=H(1)$). This proves the desired claim. 

(ii) This is essentially the same argument as used in \cite[Theorem 5.1]{W}. Let $f_5 \in \mathcal{M}_*^!$ be a modular form of minimal weight $k_5$ that does not lie in the subring generated by $f_0,...,f_4$, and compute the determinant by cofactor expansion: $$0 = \mathrm{det} \begin{pmatrix} k_0 f_0 & ... & k_4 f_4 & k_5 f_5 \\ k_0 f_0 & ... & k_4 f_4 & k_5 f_5 \\ \nabla f_0 & ... & \nabla f_4 & \nabla f_5 \end{pmatrix} = \sum_{i=0}^5 (-1)^i k_i f_i J_i, \; \; J_i = J(f_0,...,\hat f_i,...,f_5).$$ Each $J_i$ vanishes on $D(1)\setminus H$ and by assumption (ii) is a multiple of $J$ in the ring $\mathcal{M}_*^!$, say $J_i = J \cdot g_i$; and of course $g_5$ is $1$. This yields the representation $$f_5 = \sum_{i=0}^4 \frac{(-1)^i k_i}{k_5} f_i g_i.$$ Each $g_i$ lies in $\mathbb{C}[f_0,...,f_4]$ by minimality of $f_5$; but then $f_5$ also lies in $\mathbb{C}[f_0,...,f_4]$, a contradiction.
\end{proof}

\begin{remark} Theorem \ref{th:j} generalizes in an obvious way to meromorphic modular forms on orthogonal groups of lattices of higher rank. We will construct some free algebras of meromorphic modular forms on lattices of higher rank in a separate paper. 
\end{remark}

\begin{remark} If $f_0,...,f_4$ freely generate the algebra of symmetric modular forms, then any skew-symmetric modular form (that is, a modular form with the determinant character) vanishes on $D(1)\setminus H$ as in the proof of Theorem \ref{th:j}, and is therefore a multiple of $J = J(f_0,...,f_4)$ under the assumption in Theorem \ref{th:j} (ii). It follows that the full ring of meromorphic modular forms is $$\mathbb{C}[f_0,...,f_4, J] / (J^2 - P(f_0,...,f_4))$$ for a uniquely determined polynomial $P$.
\end{remark}

\section{Modular forms with poles on the Siegel half-space}

In this section we consider the simplest possible hyperplane arrangement in $\mathbf{H}_2$: the Siegel upper half-space $\mathbb{H}_2$ together with its images under $\mathrm{SU}_{2,2}(\mathcal{O}_K)$. The intersection of any two hyperplanes in this arrangement determines a Heegner divisor in $\mathbb{H}_2$ whose irreducible components have discriminant at most $1 / |d_K|$, and is therefore empty whenever $d_K \notin \{-3, -4\}$. By a separate computation one can check that the same claim holds when $d_K = -4$. (It is not satisfied when $d_K = -3$; indeed, in this case, there is a cusp form of weight $9$ that vanishes only on the orbits of $\mathbb{H}_2$, cf. \cite{DK1}.)

We will show by cases that the rings of meromorphic Hermitian modular forms with poles confined to these hypersurfaces are, when the underlying number field has discriminant $-4,-7,-8$ or $-11$, freely generated by forms of weights $2,4,6,8,10$; and $2,3,4,7,8$; and $2,3,4,6,8$; and $2,3,4,5,6$ respectively. Finally we will prove that the ring in question cannot be generated by only five forms for any other discriminant.

\subsection{Discriminant $-4$}

Symmetric Hermitian modular forms for the group $\Gamma = \mathrm{SU}_{2, 2}(\mathbb{Z}[i])$ correspond to modular forms for the discriminant kernel of the root lattice $L = 2U \oplus 2A_1$, i.e. $\mathbb{Z}^6$ with Gram matrix $\mathbf{S} = \begin{psmallmatrix} 0 & 0 & 0 & 0 & 0 & 1 \\ 0 & 0 & 0 & 0 & 1 & 0 \\ 0 & 0 & 2 & 0 & 0 & 0 \\ 0 & 0 & 0 & 2 & 0 & 0 \\ 0 & 1 & 0 & 0 & 0 & 0 \\ 1 & 0 & 0 & 0 & 0 & 0 \end{psmallmatrix}$. Through this correspondence the $\Gamma$-orbit of the Siegel half-space $\mathbb{H}_2 = \{z \in \mathbf{H}_2: \; z^T = z\}$ is an irreducible Heegner divisor of discriminant $1/4$ lines: $$H := H(1/4, (0, 0, 1/2, 0, 0, 0)) = \bigcup \{\lambda^{\perp}: \; \lambda \in (0,0,1/2,0,0,0) + \mathbb{Z}^6; \; \lambda^T \mathbf{S} \lambda = 1/2\}.$$

\begin{remark} There are two $\Gamma$-orbits of discriminant $1/4$ divisors, i.e. $H$ and $H(1/4, (0, 0, 0,  1/2, 0, 0))$. One is represented by the Siegel half-space $\mathbb{H}_2$, and the other by $$\{z \in \mathbf{H}_2: \; \mathrm{tr}(\begin{psmallmatrix} 0 & 1/2 \\ 1/2 & 0 \end{psmallmatrix} z) = 0\} = \{z \in \mathbf{H}_2: \; z_1 = -z_2\}.$$ Under the full Hermitian modular group $\mathrm{U}_{2,2}(\mathbb{Z}[i])$ these orbits coincide. It is crucial that we consider only the subgroup $\Gamma$ because there are modular forms holomorphic away from $\mathbf{H}_2 \backslash (\Gamma \cdot \mathbb{H}_2 \cup \Gamma \cdot \{z: \; z_1 = -z_2\})$ of negative weight such as the form $\phi_4^{-1}$ below, and therefore the spaces of these modular forms are not finite-dimensional. The ring of holomorphic modular forms for this group was determined by Dern and Krieg \cite{DK1}, building on earlier work of Freitag \cite{F}.
\end{remark}

The Weil representation attached to $\mathbf{S}$ admits the weight $(-1)$ nearly-holomorphic vector-valued modular form
\begin{align*} F_{-1}(\tau) &= \frac{8 \eta(2\tau)^{14}}{\eta(\tau)^{12} \eta(4\tau)^4} e_0 +  \frac{\eta(\tau)^2}{\eta(2\tau)^4} \sum_{\substack{v \in L'/L \\ Q(v) = 1/4}} e_v - \frac{32 \eta(2\tau)^2 \eta(4\tau)^4}{\eta(\tau)^8} \sum_{\substack{v \in L'/L \\ Q(v) = 1/2}} e_v,
\end{align*}
where as usual $\eta(\tau) = q^{1/24} \prod_{n=1}^{\infty} (1 - q^n)$. This is mapped under the Borcherds lift to a weight 4 modular form $\phi_4$ with quadratic character and simple zeros on both $\Gamma$-orbits of discriminant $1/4$ hyperplanes (i.e. on the Heegner divisor $H(1/4)$), as well as a nearly-holomorphic modular form
$$G_{-1}(\tau) = q^{-1} e_0 + 68 e_0 + 4928 q^{1/2} e_{(0,0,1/2,1/2,0,0)} + ...$$ which lifts to a skew-symmetric Borcherds product $\Phi_{34}$ of weight $34$ with simple zeros exactly on the Heegner divisor $D(1)=H(1)\cup H(1/4)$. In weight one we have (up to scalar multiple) a unique nearly-holomorphic modular form $F_1(\tau)$ whose image under the singular additive lift is a weight two meromorphic modular form $\phi_2$ with double poles only on $H$: $$F_1(\tau) = q^{-1/4} e_{(0,0,1/2,0,0,0)} - 2 e_{(0,0,0,0,0,0)} + 56 q^{1/2} e_{(0,0,1/2,1/2,0,0) } + ...$$

Besides $\phi_2, \phi_4$ we also need the Hermitian Eisenstein series $\cE_4, \cE_6, \cE_{10}$ of weights $4, 6, 10$.

\begin{theorem} $$\mathcal{M}_*^! = \mathbb{C}[\phi_2, \cE_4, \cE_6, \phi_4^2, \cE_{10}].$$
\end{theorem}
\begin{proof} The forms $\phi_2, \cE_4, \cE_6, \phi_4, \cE_{10}$ are algebraically independent, because: the images of $\phi_2, \cE_4, \cE_6$ and an appropriate linear combination of $\cE_4\cE_6$ and $\cE_{10}$ under the pullback map to $(0, 0, 0, 1/2, 0, 0)^{\perp}$ are the algebraically independent Siegel modular forms of weight two (with a double pole on the diagonal) and weights 4, 6, 10, and because $\phi_4$ has a zero there.

Now the Jacobian $J = J(\phi_2, \cE_4, \cE_6, \phi_4^2, \cE_{10})$ is nonzero, holomorphic away from the divisor $H$, has weight $$\mathrm{wt}(J) = 2 + 4 + 6 + 8 + 10 + 4 = 34,$$ and it vanishes on $D(1)\setminus H = H(1,0)\cup \Gamma\cdot \{z: \; z_1 = -z_2\}$, so $J / \Phi_{34} \in \mathcal{M}_0^! = \mathbb{C}$. The claim follows after applying Theorem \ref{th:j}.
\end{proof}

\begin{remark} In the notation of \cite{DK1} the meromorphic form $\phi_2$ is the form $$\phi_2 = \frac{a \cE_4 \cE_6 + b \cE_{10} + c \phi_{10}}{\phi_4^2}$$ where $a\cE_4 \cE_6 + b \cE_{10} + c \phi_{10}$ is the unique linear combination that vanishes on $(0, 0, 0, 1/2, 0, 0)^{\perp}$.
\end{remark}

\subsection{Discriminant $-7$}

The signature $(4, 2)$ lattice whose modular forms correspond to degree-two modular forms for the group $\Gamma = \mathrm{SU}_{2,2}(\mathcal{O}_K)$ with $K = \mathbb{Q}(\sqrt{-7})$ is $L = \mathbb{Z}^6$ with Gram matrix $\begin{psmallmatrix} 0 & 0 & 0 & 0 & 0 & 1 \\ 0& 0 & 0 & 0 & 1 & 0 \\ 0 & 0 & 2 & 1 & 0 & 0\\ 0 & 0 & 1 & 4 & 0 & 0 \\ 0 & 1 & 0 & 0 & 0 & 0 \\ 1 & 0 & 0 & 0 & 0 & 0 \end{psmallmatrix}$. Generators and relations for the ring of holomorphic modular forms were determined in \cite{bw2}. We will compute the ring of meromorphic modular forms holomorphic away from the Siegel upper half-space: $$\mathcal{M}_*^! = \{f \in \mathcal{A}_k, \; f \; \text{holomorphic on} \; \mathbf{H}_2 \backslash (\Gamma \cdot \mathbb{H}_2)\}.$$

Since the discriminant is prime, we can construct input forms into the singular lift from nearly-holomorphic modular forms in $M_*^{!, -}(\Gamma_0(7), \chi)$ using the Bruinier--Bundschuh correspondence \cite{BB} (see also the constructions in \cite{bw2}). There are Borcherds products $\phi_7$ and $\Phi_{28}$ obtained from the weight $(-1)$ forms $$F_{-1}(\tau) = 2q^{-2} + 6q^{-1} + 14 - 38q^3 - 96q^5 \pm ...$$ and $$G_{-1}(\tau) = q^{-7} + 14q^{-1} + 56 + 4522q^3 + 27846q^5 + ...$$ with divisors $$\mathrm{div}\, \phi_7 = 3 H(1/7) + H(2/7), \;\; \mathrm{div}\, \Phi_{28} = 7 H(1/7) + H(1),$$ and there is a weight $1$ form $$F_1(\tau) = 2q^{-1} - 2 + 50q^3 - 52q^5 \pm ...$$ that lifts to a meromorphic Hermitian modular form $\phi_2$ with order two poles exactly on $\Gamma \cdot \mathbb{H}_2$. Finally we need the weight $2$ vector-valued modular form $$F_2(\tau) = q^{-1/7} (e_{(0, 0, 1/7, 5/7, 0, 0)} - e_{(0,0,6/7,2/7,0,0)}) + 13 q^{3/7}(e_{(0,0,5/7,4/7,0,0)} - e_{(0,0,2/7,3/7,0,0)}) + O(q^{5/7})$$ which lifts to a meromorphic form $\phi_3$ with triple poles on $\Gamma \cdot \mathbb{H}_2$.

In the notation of \cite{bw2}, $$\phi_2 = \frac{m_9}{b_7}, \; \phi_3 = \frac{m_{10}^{(2)}}{b_7}, \; \phi_7 = b_7.$$

Let $\cE_4$ and $\cE_8$ be the Eisenstein series of weight $4$ and $8$ for $\Gamma=\mathrm{SU}_{2,2}(\mathcal{O}_K)$ respectively.

\begin{theorem} $$\mathcal{M}_*^! = \mathbb{C}[\phi_2, \phi_3, \cE_4, \phi_7, \cE_8].$$
\end{theorem}
\begin{proof} The algebraic independence of $\phi_2, \phi_3, \cE_4, \phi_7, \cE_8$ follows immediately from that of the (holomorphic) modular forms $\phi_2 \phi_7, \phi_3 \phi_7, \cE_4, \phi_7, \cE_8$. The Jacobian $J(\cE_4, \phi_7, \cE_8, \phi_2 \phi_7, \phi_3 \phi_7)$ can be shown to be nonzero by direct computation (for this one needs at least the first $8$ Fourier--Jacobi coefficients), or the algebraic independence can be derived from the results of \cite{bw2}. Then $$J = J(\phi_2, \phi_3, \cE_4, \phi_7, \cE_8)$$ has weight $28$ and vanishes on the Heegner divisor $H(1)$, so $J / \Phi_{28} \in \mathcal{M}_0^! = \mathbb{C}$. The claim follows from Theorem \ref{th:j}.
\end{proof}
This implies that, up to a nonzero constant multiple, the Jacobian of the holomorphic forms above is $$J(\cE_4, \phi_7, \cE_8, \phi_2 \phi_7, \phi_3 \phi_7) = \phi_7^2 J(\phi_2, \phi_3, \cE_4, \phi_7, \cE_8) = \phi_7^2 \Phi_{28} \in \mathcal{M}_{42}.$$

\subsection{Discriminant $-8$}

The lattice $L$ we take in this section has Gram matrix $\begin{psmallmatrix} 0 & 0 & 0 & 0 & 0 & 1 \\ 0 & 0 & 0 & 0 & 1 & 0 \\ 0 & 0 & 2 & 0 & 0 & 0 \\ 0 & 0 & 0 & 4 & 0 & 0 \\ 0 & 1 & 0 & 0 & 0 & 0 \\ 1 & 0 & 0 & 0 & 0 &0 \end{psmallmatrix}$. We will use the nearly-holomorphic modular forms of weight $(-1)$ whose Fourier expansions begin $$F_{-1}(\tau) = q^{-1/4} e_{(0,0,1/2,0,0,0)} + 2 q^{-1/8} (e_{(0,0,0,1/4,0,0)} + e_{(0,0,0,3/4,0,0)}) + 6e_{(0,0,0,0,0,0)} + ...$$ $$G_{-1}(\tau) = q^{-1/2} e_{(0,0,0,1/2,0,0)} - 4q^{-1/8} (e_{(0,0,0,1/4,0,0)} + e_{(0,0,0,3/4,0,0)}) + 6 e_{(0,0,0,0,0,0)} + ...$$ $$H_{-1}(\tau) = q^{-1} e_{(0,0,0,0,0,0)} + 8q^{-1/8} (e_{(0,0,0,1/4,0,0)} + e_{(0,0,0,3/4,0,0)}) + 54 e_{(0,0,0,0,0,0)} + ...$$ These lift to Borcherds products $\psi_3, \phi_3, \Phi_{27}$ with the following properties: \\
(1) $\psi_3$ is a weight $3$ cusp form with a quadratic character and divisor $$\mathrm{div}\, \psi_3 = 2 H(1/8) + H(1/4);$$
(2) $\phi_3$ is a meromorphic weight $3$ form with trivial character and divisor $$\mathrm{div}\, \phi_3 = -3 H(1/8) + H(1/2);$$
(3) $\Phi_{27}$ is a skew-symmetric cusp form of weight $27$ with divisor $$\mathrm{div}\, \Phi_{27} = 8H(1/8) + H(1/4) + H(1).$$
The form $\phi_3$ can also be constructed as a singular additive lift. We also need the meromorphic form $\phi_2$ defined as the singular additive lift of the weight $1$ form $$F_1(\tau) = q^{-1/8} (e_{(0,0,0,1/4,0,0)} + e_{(0,0,0,3/4,0,0)}) - 2 e_{(0,0,0,0,0,0)} + ... \in M_1^!(\rho).$$

Let $\mathcal{M}_*^!$ be the ring of meromorphic modular forms that are holomorphic on $\mathbf{H}_2 \backslash (\Gamma \cdot \mathbb{H}_2)$.

\begin{theorem} $$\mathcal{M}_*^! = \mathbb{C}[\phi_2, \phi_3, \cE_4, \psi_3^2, \cE_8]$$
\end{theorem}
\begin{proof} The product $\phi_2 \psi_3^2$ is the unique weight $8$ cusp form in the Maass subspace (cf. \cite{K}) that vanishes on the Heegner divisor $H(1/4)$; this is the additive lift of the vector-valued cusp form $$q^{1/4} e_{(0,0,1/2,1/2,0,0)} - 2q^{1/2} e_{(0,0,0,1/2,0,0)} + 8q^{5/8} (e_{(0,0,1/2,1/4,0,0)} + e_{(0,0,1/2,3/4,0,0)}) + ... \in S_7(\rho).$$ Using the Fourier--Jacobi expansions (with at least the first $8$ coefficients) of the holomorphic forms $\phi_2 \psi_3^2,$ $\phi_3 \psi_3^2, \cE_4, \psi_3^2, \cE_8$ we find that their Jacobian is not identically zero. It follows that $$J = J(\phi_2, \phi_3, \cE_4, \psi_3^2, \cE_8) \in \mathcal{M}_{27}^!$$ is nonzero and vanishes on the reflective divisor $D(1)=H(1)+H(1/4)$, so $J / \Phi_{27} \in \mathcal{M}_0^! = \mathbb{C}$ and the claim follows from Theorem \ref{th:j}.
\end{proof}

\begin{remark} The ring structure of holomorphic modular forms for this group was determined by Dern and Krieg in \cite{DK2}.
\end{remark}

\subsection{Discriminant $-11$}

We use the lattice $L$ with Gram matrix $\begin{psmallmatrix} 0 & 0 & 0 & 0 & 0 & 1 \\ 0 & 0 & 0 & 0 & 1 & 0 \\ 0 & 0 & 2 & 1 & 0 & 0 \\ 0 & 0 & 1 & 6 & 0 & 0 \\ 0 & 1 & 0 & 0 & 0 & 0 \\ 1 & 0 & 0 & 0 & 0 &0 \end{psmallmatrix}$. Since the discriminant is prime, we again have the Bruinier--Bundschuh isomorphism between vector-valued modular forms and the minus-space of modular forms of level $\Gamma_0(11)$ with the quadratic character. In weight $(-1)$ there are nearly-holomorphic modular forms with Fourier expansion beginning $$F_{-1}(\tau) = 2q^{-3} + 10q^{-1} + 10 - 12q^2 - 40q^6 \pm ...$$ $$G_{-1}(\tau) = 2q^{-4} - 8q^{-1} + 6 + 46q^2 - 150q^6 \pm ...$$ $$H_{-1}(\tau) = q^{-11} + 22q^{-1} + 48 + 528q^2 + 7920q^6 + ...$$
They lift to Borcherds products $\phi_5, \phi_3, \Phi_{24}$ with the properties\\
(1) $\phi_5$ is a weight $5$ cusp form with divisor $$\mathrm{div}\, \phi_5 = 5H(1/11) + H(3/11);$$
(2) $\phi_3$ is a meromorphic weight $3$ form with divisor $$\mathrm{div}\, \phi_3 = -3H(1/11) + H(4/11);$$
(3) $\Phi_{24}$ is a skew-symmetric cusp form of weight $24$ with divisor $$\mathrm{div}\, \Phi_{24} = 11 H(1/11) + H(1).$$

We will also use the singular additive lift $\phi_2$ of the weight $1$ form $$F_1(\tau) = q^{-1} - 1 + 7q^2 + 7q^6 + 19q^7 \pm ...$$ which is holomorphic except for double poles on $\Gamma \cdot \mathbb{H}_2$. Define $\mathcal{M}_*^!$ to consist of meromorphic modular forms holomorphic on $\mathbf{H}_2  \backslash (\Gamma \cdot \mathbb{H}_2)$.

\begin{theorem} $$\mathcal{M}_*^! = \mathbb{C}[\phi_2, \phi_3, \cE_4, \phi_5, \cE_6].$$
\end{theorem}
\begin{proof} By a similar argument to the previous section, $\phi_2 \phi_5$ is the unique weight $7$ cusp form in the Maass space. A direct computation using the first 8 Fourier--Jacobi coefficients shows that the Jacobian of the (holomorphic) modular forms $\phi_2 \phi_5, \phi_3 \phi_5, \cE_4, \phi_5, \cE_6$ is not identically zero, which implies that $\phi_2, \phi_3, \cE_4, \phi_5, \cE_6$ are algebraically independent. Their Jacobian $$J = J(\phi_2, \phi_3, \cE_4, \phi_5, \cE_6)$$ has weight $24$ and vanishes on $H(1)$, so $J / \Phi_{24} \in \mathcal{M}_0^! = \mathbb{C}$ and the claim follows from Theorem \ref{th:j}.
\end{proof}

\begin{remark} The ring of holomorphic modular forms was computed in \cite{bw2}.
\end{remark}

\subsection{A nonexistence theorem} 

We will show that the four cases above account for all freely generated algebras of Hermitian modular forms that are holomorphic away from $\mathbb{H}_2$. Let $K$ be an imaginary-quadratic number field of discriminant $d_K$ and let $\mathcal{M}_k^!$ be the space of meromorphic Hermitian modular forms of weight $k$ whose poles lie only on conjugates of $\mathbb{H}_2$.

\begin{lemma}\label{lem:dim1} $\mathrm{dim}\, \mathcal{M}_1^! \le 1$.
\end{lemma}
\begin{proof} The Fourier--Jacobi expansion of any holomorphic Hermitian modular form $h$ takes the form $$h \Big( \begin{psmallmatrix} \tau & z_1 \\ z_2 & w \end{psmallmatrix} \Big) = \sum_{n=0}^{\infty} h_n(\tau, z_1, z_2) e^{2\pi i n w}.$$ In particular, if $h$ has weight one, then for any $\lambda \in \mathcal{O}_K$, the form $h_n(\tau, \lambda z, \overline{\lambda} z)$ is a Jacobi form of weight one (and index $N_{K/\mathbb{Q}} \lambda$) and vanishes identically by a theorem of Skoruppa (cf. \cite{EZ}, Theorem 5.7). Therefore the zeros of each $h_n$ are dense in $\mathbb{H} \times \mathbb{C}^2$ so all $h_n$ vanish identically. This shows that there are no holomorphic Hermitian modular forms of weight one (without character) for any discriminant.

Now suppose $f, g \in \mathcal{M}_1^!$ are linearly independent. Since a weight one form in $\mathcal{M}_1^!$ can have at worst a simple pole on $\mathbb{H}_2$ with constant residue, some linear combination of $f$ and $g$ must be holomorphic on $\mathbb{H}_2$ and therefore holomorphic everywhere. This cannot happen by the previous paragraph.
\end{proof}

\begin{theorem}\label{th:classify} Suppose $|d_K| > 11$. Then $\mathcal{M}_*^!$ cannot be generated by only five modular forms.
\end{theorem}
\begin{proof} Suppose $\mathcal{M}_*^!$ is generated by five modular forms $f_1,...,f_5$ of weights $k_1,...,k_5$. Since the intersection of any two hyperplanes in the arrangement is disjoint from $\mathbb{H}_2$, the boundary components of the Looijenga compactification have dimension at most one. Using an argument analogous to \cite[Theorem 3.5 (2)]{W}, we find that the Jacobian $$J = J(f_1,...,f_5)$$ has only a simple zero on the Heegner divisor 
$$D(1)=\bigoplus_{\lambda\in L, Q(\lambda)=1} \mathcal{D}_\lambda(L),$$ a (not necessarily simple) zero or pole on the Siegel upper half-space, and no zeros or poles otherwise. By the Bruinier converse theorem \cite{Br} it is the Borcherds lift of a vector-valued modular form $F \in M_{-1}^!(\rho)$. We can fix $\rho$ to be the Weil representation of the lattice $\mathcal{O}_K$, as this yields the same Weil representation as $2U \oplus \mathcal{O}_K$. The form $F$ then has principal part $$F(\tau) = q^{-1} e_{0} + m q^{-1/|d_K|} (e_v + e_{-v}) + 2 \cdot \mathrm{wt}(J) e_0 + O(q^{1/|d_K|})$$ for some $m \in \mathbb{Z}$, (the multiplicity of $\mathbb{H}_2$ in the divisor of $J$), where $v = \frac{i}{\sqrt{|d_K|}} \in \mathcal{O}_K^{\#} / \mathcal{O}_K$.

Consider the theta-contraction $\vartheta F$ to the sublattice $\mathbb{Z} \subseteq \mathcal{O}_K$, which corresponds to the quasi-pullback of $J$ to $\mathbb{H}_2$. The Looijenga condition implies that $\vartheta F \in M_{-1/2}^!(\rho_{\mathbb{Z}})$ is a modular form with principal part $q^{-1} e_{0} + O(q^0)$ at infinity, which determines it uniquely: $$\vartheta F(\tau) = (q^{-1} + 70 + 131976q + ...) e_{0} + (32384q^{3/4} + ...) e_{1/2}.$$ Since the weight of $J$ increases or decreases by the multiplicity $m$ under the quasi-pullback to $\mathbb{H}_2$, we find $$35 = \mathrm{wt}(J) + m.$$

The dual representation $\rho^*$ admits a modular form of weight three, which can be constructed as the Serre derivative of the usual theta series: $$G(\tau) := \frac{1}{2\pi i} \Theta'(\tau) - \frac{1}{12} E_2(\tau) \Theta(\tau),$$ where $\Theta(\tau) = \sum_{\lambda \in \mathcal{O}_K^{\#}} q^{N_{K/\mathbb{Q}}(\lambda)} e_{\lambda}$. In particular the Fourier expansion of $G$ has the form $$G(\tau) = -\frac{1}{12} e_0 + \Big( \frac{1}{|d_K|} - \frac{1}{12} \Big) q^{1 / |d_K|} (e_v + e_{-v}) \pm ... \pm \frac{23}{6} q e_0 + ...$$ for a unique pair of cosets $\pm v \in \mathcal{O}_K^{\#} / \mathcal{O}_K$. If we write $$F(\tau) = \sum_{x \in \mathcal{O}_K^{\#} / \mathcal{O}_K} f_x(\tau) e_x \; \text{and} \; G(\tau) = \sum_{x \in \mathcal{O}_K^{\#} / \mathcal{O}_K} g_x(\tau) e_x$$ then $\sum_{x \in \mathcal{O}_K^{\#}/\mathcal{O}_K} f_x(\tau) g_x(\tau)$ is a nearly-holomorphic modular form of weight two and therefore has constant term zero. This yields the identity $$\frac{23}{6} - \frac{70 - 2m}{12} + 2m \Big( \frac{1}{|d_K|} - \frac{1}{12} \Big) = 0$$ which implies $m = |d_K|$. This relation can also be proved using the approach to classify reflective modular forms in \cite{W1}. 

Since $k_1,...,k_5 \ge 1$ and by Lemma \ref{lem:dim1} at most one of the generators can have weight one, we have $$35 - |d_K| = \mathrm{wt}(J) = k_1+k_2+k_3+k_4+k_5+4 \ge 1+2+2+2+2+4 = 13$$ and therefore $|d_K| \le 22$. By applying Borcherds' obstruction theorem \cite{B2} to the lattices with $|d_K| \le 22$ we see that the desired vector-valued form $F$ does not exist unless $|d_K| \le 11$.
\end{proof}

\section{Rings of modular forms with multiple poles}

We will consider rings of meromorphic Hermitian modular forms with poles on Heegner divisors that are not necessarily conjugate to the Siegel upper half-space. For the discriminants $|d_K| \in \{15, 19, 20, 24\}$ we find additional hyperplane arrangements $H$ such that the rings of meromorphic modular forms that are holomorphic away from $H$ are polynomial algebras without relations.

\subsection{Discriminant $-15$}

The lattice $L$ has Gram matrix $\begin{psmallmatrix} 0 & 0 & 0 & 0 & 0 & 1 \\ 0 & 0 & 0 & 0 & 1 & 0 \\ 0 & 0 & 2 & 1 & 0 & 0 \\ 0 & 0 & 1 & 8 & 0 & 0 \\ 0 & 1 & 0 & 0 & 0 & 0 \\ 1 & 0 & 0 & 0 & 0 &0 \end{psmallmatrix}$. There are two $\pm$-pairs of orbits of norm $1/15$ vectors under the discriminant kernel, so the discriminant $1/15$ Heegner divisor splits into two irreducible components: $$H(1/15) = H(1/15, (0, 0, 1/15, -2/15, 0, 0)) + H(1/15, (0, 0, 4/15, -8/15, 0, 0)).$$ There are weight $(-1)$ nearly-holomorphic vector-valued modular forms with Fourier expansions beginning \begin{align*} F_{-1}(\tau) &= q^{-4/15} (e_{(0,0,2/15,-4/15,0,0)} + e_{(0,0,-2/15,4/15,0,0)}) \\ &+ 7 q^{-1/15} (e_{(0,0,4/15,-8/15,0,0)} + e_{(0,0,-4/15,8/15,0,0)}) \\ &- 4q^{-1/15} (e_{(0,0,1/15,-2/15,0,0)} + e_{(0,0,-1/15,2/15,0,0)}) \\ &+ 6e_{(0,0,0,0,0,0)} + ... \end{align*}
\begin{align*} G_{-1}(\tau) &= q^{-4/15} (e_{(0,0,-8/15, 1/15,0,0)} + e_{(0,0,8/15,-1/15,0,0)}) \\ &+ 7q^{-1/15}(e_{(0,0,1/15,-2/15,0,0)} + e_{(0,0,-1/15,2/15,0,0)}) \\ & - 4q^{-1/15} (e_{(0,0,4/15,-8/15,0,0)} + e_{(0,0,-4/15,8/15,0,0)}) \\ &+ 6e_{(0,0,0,0,0,0)} + ... \end{align*}
\begin{align*} H_{-1}(\tau) &= q^{-1} e_{(0,0,0,0,0,0)} \\ &+15q^{-1/15}(e_{(0,0,1/15,-2/15,0,0)} + e_{(0,0,-1/15,2/15,0,0)} + e_{(0,0,4/15,-8/15,0,0)} + e_{(0,0,-4/15,8/15,0,0)}) \\ &+ 40 e_{(0,0,0,0,0,0)} + ... \end{align*} that lift to meromorphic Borcherds products, respectively labelled $\phi_3, \psi_3, \Phi_{20}$ of weights $3, 3, 20$. The form $\Phi_{20}$ is a skew-symmetric cusp form and the forms $\phi_3, \psi_3$ have triple poles along one of the two components of $H(1/15)$ and are holomorphic elsewhere.

We additionally need the following additive lifts with singularities: the form $\phi_2$ which is the additive lift of the weight-one nearly holomorphic form $$F_1(\tau) = q^{-1/15} (e_{(0,0,1/15,-2/15,0,0)} + e_{(0,0,-1/15,2/15,0,0)} + e_{(0,0,4/15,-8/15,0,0)} + e_{(0,0,-4/15,8/15,0,0)}) - 2 e_{(0,0,0,0)} + ...$$ and the forms $\phi_4$, $\psi_4$ coming from the weight-three nearly holomorphic forms $$F_3(\tau) = q^{-1/15} (e_{(0,0,-1/15,2/15,0,0)} + e_{(0,0,1/15,-2/15,0,0)}) - 5q^{1/3} (e_{(0,0,1/3,1/3,0,0)} + e_{(0,0,2/3,2/3,0,0)}) + ...$$
$$G_3(\tau) = q^{-1/15} (e_{(0,0,-4/15,8/15,0,0)} + e_{(0,0,4/15,-8/15,0,0)}) - 5q^{1/3} (e_{(0,0,1/3,1/3,0,0)} + e_{(0,0,2/3,2/3,0,0)}) + ...$$

The ring of meromorphic modular forms $\mathcal{M}_*^!$ consists of modular forms which are holomorphic away from the divisor $$H(1/15) = H(1/15, (0,0,1/15,-2/15,0,0)) + H(1/15, (0,-1,4/15,-8/15,1,0)).$$

\begin{theorem} $$\mathcal{M}_*^! = \mathbb{C}[\phi_2, \phi_3, \psi_3, \phi_4, \psi_4].$$
\end{theorem}
\begin{proof} We will outline an argument to show that these meromorphic forms are algebraically independent that avoids computing any Jacobians directly. Consider the succesive restrictions to subgrassmannians in the following chain of signature $(n, 2)$ sublattices: $$L \rightarrow (\mathbb{Z}v_1 + \mathbb{Z} v_2 + \mathbb{Z}v_3 + \mathbb{Z}v_4 + \mathbb{Z}v_5 ) \rightarrow (\mathbb{Z}v_1 + \mathbb{Z}v_2 + \mathbb{Z}v_3 + \mathbb{Z} v_4) \rightarrow (\mathbb{Z} v_1 + \mathbb{Z}v_2 + \mathbb{Z}v_3)$$ where $$v_1 = (1,0,0,0,0,0), \; v_2 = (0,2,0,-1,-3,0), \; v_3 = (0,0,0,0,0,1), \; v_4=(0,0,0,0,1,0), \; v_5 = (0,0,1,0,0,0).$$ The lattices above are constructed as the intersections $u_1^{\perp}$, $u_1^{\perp} \cap u_2^{\perp}$ and $u_1^{\perp} \cap u_2^{\perp} \cap u_3^{\perp}$ where $$u_1 = (0,0,2/15,-4/15,-1,0), \; u_2 = (0,0,8/15,-1/15,0,0), \; u_3 = (0,1,1,0,2,0) \in L'.$$
The rank three lattice in the end has Gram matrix $\begin{psmallmatrix} 0 & 0 & 1 \\ 0 & -4 & 0 \\ 1 & 0 & 0 \end{psmallmatrix}$ and the modular forms on its orthogonal group are simply elliptic modular forms of level $\Gamma_0(2)$ (and double the starting weight). The images of the generators under these restriction maps (denoted $P$) are computed using theta-contraction as suggested in section 2.1:
\begin{align*} 
&\phi_2& &\mapsto& &P\phi_2& &\mapsto& &P^2 \phi_2& &\mapsto& &1+48q+624q^2+1344q^3 + ... \\
&\phi_3& &\mapsto& &0&  &\mapsto& &0& &\mapsto& &0 \\
&\psi_3& &\mapsto&  &P\psi_3& &\mapsto& &0& &\mapsto& &0 \\
&\phi_4& &\mapsto& &P\phi_4& &\mapsto& &P^2 \phi_4& &\mapsto& &q - 8q^2 + 12q^3 \pm ...\\
&\psi_4& &\mapsto& &P\cE_4& &\mapsto& &P^2 \cE_4& &\mapsto& &q - 8q^2 + 12q^3 \pm... 
\end{align*}
In particular one finds $P\phi_3 = 0$; $P \psi_3 \ne 0$ but $P^2 \psi_3 = 0$; and one also finds $P^2 (\phi_4 - \psi_4) \ne 0$ but $P^3(\phi_4 - \psi_4) = 0$. (The vanishing of the Borcherds products is clear from their divisors.) This implies successively that the sets $\{\phi_2, \phi_4\}$, $\{\phi_2, \phi_4, \psi_4\}$, $\{\phi_2, \phi_4, \psi_4, \psi_3\}$ and finally $\{\phi_2, \phi_4, \psi_4, \psi_3, \phi_3\}$ are algebraically independent.

Similarly to the previous cases, we find that the Jacobian of the five forms is a nonzero multiple of $\Phi_{20}$ and therefore that these forms are generators of $\mathcal{M}_*^!$.
\end{proof}

\begin{remark} The \emph{maximal discrete extension} of the Hermitian modular group \cite{KRW} contains Atkin--Lehner involutions which swap the two pairs of Heegner divisors of discriminant $1/15$, and in particular swap the pairs $\{\phi_3, \psi_3\}$ and $\{\phi_4, \psi_4\}$.
\end{remark}

\subsection{Discriminant $-19$}
In this case $L$ is the lattice with Gram matrix $\begin{psmallmatrix} 0 & 0 & 0 & 0 & 0 & 1 \\ 0 & 0 & 0 & 0 & 1 & 0 \\ 0 & 0 & 2 & 1 & 0 & 0 \\ 0 & 0 & 1 & 10 & 0 & 0 \\ 0 & 1 & 0 & 0 & 0 & 0 \\ 1 & 0 & 0 & 0 & 0 & 0 \end{psmallmatrix}$.
We will again construct modular forms using the Bruinier--Bundschuh isomorphism (cf. section 3.2) from nearly-holomorphic modular forms of level $\Gamma_0(19)$ and quadratic character $\chi$ lying in the minus-space. There are forms of weight $(-1)$ with Fourier expansions beginning $$F_{-1}(\tau) = 2q^{-6} + 2q^{-4} - 8q^{-1} + 6 + 12q^2 \pm ...$$
$$G_{-1}(\tau) = 2q^{-5} + 2q^{-4} + 10q^{-1} + 8 - 6q^2 \pm ...$$
$$H_{-1}(\tau) = 2q^{-7} + 10q^{-1} + 10 + 20q^2 \pm ...$$
$$J_{-1}(\tau) = q^{-19} + 2q^{-4} + 30q^{-1} + 38 + 198q^2 + ...$$
which lift to Borcherds products $\phi_3, \phi_4, \phi_5, \Phi_{19}$. The forms $\phi_4, \phi_5, \Phi_{19}$ are cusp forms of weights $4, 5, 19$ with $\Phi_{19}$ skew-symmetric, and $\phi_3$ is meromorphic with triple poles on $H(1/19)$.

We will be interested in the ring of meromorphic modular forms $\mathcal{M}_*^!$ that are holomorphic away from $$D(4/19) = H(1/19) \cup H(4/19).$$ Besides the Borcherds products above, we also need the singular additive theta lift $\phi_1$ of the weight $0$ (vector-valued) modular form $$F_0(\tau) = q^{-4/19} (e_{2v} - e_{-2v}) - 2q^{-1/19} (e_v - e_{-v}) + O(q^{2/19})$$ where $v \in L'$ can be any vector of norm $1/19$, as well as the singular lift $\phi_2$ of the modular form $$F_1(\tau) = 2q^{-1} - 2 + 6q^2 + 10q^3 \pm ... \in M_1^{!, -}(\Gamma_0(19), \chi).$$ Let $\mathcal{M}_*^!$ be the ring of meromorphic modular forms that are holomorphic away from the divisor $$D(4/19) = H(1/19) + H(4/19).$$

\begin{theorem} $$\mathcal{M}_*^! = \mathbb{C}[\phi_1, \phi_2, \phi_3, \phi_4, \phi_5].$$
\end{theorem}
\begin{proof} Using a similar argument to the discriminant $(-15)$ modular group, we consider the restrictions to subgrassmannians with respect to the following chain of sublattices: $$L \rightarrow (\mathbb{Z}v_1 + \mathbb{Z} v_2 + \mathbb{Z}v_3 + \mathbb{Z}v_4 + \mathbb{Z}v_5 ) \rightarrow (\mathbb{Z}v_1 + \mathbb{Z}v_2 + \mathbb{Z}v_3 + \mathbb{Z} v_4) \rightarrow (\mathbb{Z} v_1 + \mathbb{Z}v_2 + \mathbb{Z}v_3)$$ where $$v_1 = (1,0,0,0,0,0), \; v_2 = (0,2,0,-1,-3,0), \; v_3 = (0,0,0,0,0,1), \; v_4=(0, 1, 1, -1, -3, 0), \; v_5 = (0,1,1,0,2,0),$$ i.e. the sublattices $u_1^{\perp}$, $u_1^{\perp} \cap u_2^{\perp}$, $u_1^{\perp} \cap u_2^{\perp} \cap u_3^{\perp}$ with primitive vectors $$u_1 = (0,-1,5/19,9/19,1,0), \; u_2 = (0,0,8/19,3/19,1,0), \; u_3 = (0,0,10/19,-1/19,0,0) \in L'.$$ The lattice $\mathbb{Z}v_1 + \mathbb{Z}v_2 + \mathbb{Z}v_3$ has Gram matrix $\begin{psmallmatrix} 0 & 0 & 1 \\ 0 & -2 & 0 \\ 1 & 0 & 0 \end{psmallmatrix}$ and its modular forms are elliptic modular forms of level one of twice the starting weight. At the final stage in this restriction process the form $\phi_1$ gets a pole, so we instead consider the (holomorphic) cusp form $\phi_1^2 \phi_4$ of weight $6$. We obtain the following images under the restriction maps (again denoted $P$):
\begin{align*}
&\phi_2& &\mapsto& &P\phi_2& &\mapsto& &P^2 \phi_2& &\mapsto& &1 + 240q + 2160q^2 + ... \\
&\phi_3& &\mapsto& &0& &\mapsto& &0& &\mapsto& &0 \\
&\phi_4& &\mapsto& &P\phi_4& &\mapsto& &P^2 \phi_4& &\mapsto& &0 \\
&\phi_5& &\mapsto& &P\phi_5& &\mapsto& &0& &\mapsto& &0 \\
&\phi_1^2 \phi_4& &\mapsto& &P(\phi_1^2 \phi_4)& &\mapsto& &P^2(\phi_1^2 \phi_4)& &\mapsto& &q - 24q^2 \pm ...
\end{align*}
The point at which the products $\phi_3, \phi_4, \phi_5$ vanish in this process can be read immediately off of the principal part of their input forms because $u_1$ has norm $6/19$, $u_2$ has norm $7/19$ and $u_3$ has norm $5/19$. The level one forms $E_4(\tau)$ and $\Delta(\tau)$ in the rightmost column are algebraically independent, so this is also true for $\{\phi_2, \phi_1^2 \phi_4\}$. By considering the point at which zeros appear in this process one finds successively that the sets $\{\phi_2, \phi_1^2 \phi_4, \phi_4\}$, $\{\phi_2, \phi_1^2 \phi_4, \phi_4, \phi_5\}$, $\{\phi_2, \phi_1^2 \phi_4, \phi_4, \phi_5, \phi_3\}$ are algebraically independent.

The Jacobian of $\phi_1,\phi_2,\phi_3,\phi_4,\phi_5$ is therefore nonzero and has weight $19$. By the same argument used previously it must equal $\Phi_{19}$ and the forms generate $\mathcal{M}_*^!$.
\end{proof}

\subsection{Discriminant $-20$}

The lattice $L$ in this case has Gram matrix $\begin{psmallmatrix} 0 & 0 & 0 & 0 & 0 & 1 \\ 0 & 0 & 0 & 0 & 1 & 0 \\ 0 & 0 & 2 & 0 & 0 & 0 \\ 0 & 0 & 0 & 10 & 0 & 0 \\ 0 & 1 & 0 & 0 & 0 & 0 \\ 1 & 0 & 0 & 0 & 0 & 0 \end{psmallmatrix}$.

We will use the nearly-holomorphic modular forms of weight $(-1)$ \begin{align*} F_{-1}(\tau) &= q^{-1/2} e_{(0,0,1/2,1/2,0,0)} + q^{-1/5} (e_{(0,0,0,1/5,0,0)} + e_{(0,0,0,4/5,0,0)}) \\ &- 4q^{-1/20} (e_{(0,0,0,1/10,0,0)} + e_{(0,0,0,9/10,0,0)} + e_{(0,0,1/2,2/5,0,0)} + e_{(0,0,1/2,3/5,0,0)}) \\ &+ 6e_{(0,0,0,0,0,0)} + ... \end{align*}
\begin{align*} G_{-1}(\tau) &= q^{-3/10} (e_{(0,0,1/2,1/10,0,0)} + e_{(0,0,1/2,9/10,0,0)}) + q^{1/5}(e_{(0,0,0,1/5,0,0)} + e_{(0,0,0,4/5,0,0)}) \\ &+ 4q^{-1/20} (e_{(0,0,0,1/10,0,0)} + e_{(0,0,0,9/10,0,0)} + e_{(0,0,1/2,2/5,0,0)} + e_{(0,0,1/2,3/5,0,0)}) \\ &+ 10e_{(0,0,0,0,0,0)} + ... \end{align*}
\begin{align*} H_{-1}(\tau) &= q^{-1} e_{(0,0,0,0,0,0)} + q^{1/5}(e_{(0,0,0,1/5,0,0)} + e_{(0,0,0,4/5,0,0)}) \\ &+ 16q^{-1/20} (e_{(0,0,0,1/10,0,0)} + e_{(0,0,0,9/10,0,0)} + e_{(0,0,1/2,2/5,0,0)} + e_{(0,0,1/2,3/5,0,0)}) \\ &+ 36e_{(0,0,0,0,0,0)} + ... \end{align*}
which lift to Borcherds products $\phi_3, \phi_5, \Phi_{18}$ with trivial character. $\phi_3$ is meromorphic with triple poles and $\phi_5, \Phi_{18}$ are cusp forms. In the notation of Appendix B of \cite{bw3} these are the products $$\phi_3 = \frac{\psi_8^{(3)}}{\psi_5}, \; \phi_5 = \psi_5, \; \Phi_{18} = \psi_3 \psi_{15}.$$
We also need the singular theta lifts $\phi_1, \phi_2, \psi_3$ of weights $1, 2, 3$ of the nearly-holomorphic modular forms
\begin{align*}
F_0(\tau) &= q^{-1/5} (e_{(0,0,0,1/5,0,0)} - e_{(0,0,0,4/5,0,0)})\\
&- 2q^{-1/20} (e_{(0,0,0,1/10,0,0)} + e_{(0,0,1/2,3/5,0,0)} - e_{(0,0,0,9/10,0,0)} - e_{(0,0,1/2,2/5,0,0)}) + ...
\end{align*}
$$F_1(\tau) = q^{-1/20} (e_{(0,0,0,1/10,0,0)} + e_{(0,0,0,9/10,0,0)} + e_{(0,0,1/2,2/5,0,0)} + e_{(0,0,1/2,3/5,0,0)}) - 2 e_{(0,0,0,0,0,0)} + ...$$
$$F_2(\tau) = q^{-1/20} (e_{(0,0,1/2,2/5,0,0)} - e_{(0,0,1/2,3/5,0,0)}) + ...$$

The ring $\mathcal{M}_*^!$ will consist of meromorphic modular forms that are holomorphic away from $$D(1/5) = H(1/5) + H(1/20, (0,0,0,1/10,0,0)) + H(1/20,(0,0,1/2,2/5,0,0)).$$

\begin{theorem} $$\mathcal{M}_*^! = \mathbb{C}[\phi_1, \phi_2, \phi_3, \psi_3, \phi_5].$$
\end{theorem}
\begin{proof} Let $\lambda := (0,0,1/2,0,0,0) \in L'$. The lattice $\lambda^{\perp}$ is the root lattice $2U \oplus A_1(5)$ and the modular forms on its orthogonal group are paramodular forms of level $5$, for which generators were determined in \cite{bw1}. Among them are (up to scalar) uniquely determined cuspidal Gritsenko lifts $g_6$, $g_7$ of weights $6$ and $7$ and holomorphic Borcherds products $b_5, b_8$ of weights $5$ and $8$. We find that the images of the generators under this pullback map are
$$\phi_1 \mapsto g_6 / b_5, \; \phi_2 \mapsto g_7/b_5, \; \phi_3 \mapsto b_4^2 / b_5, \; \psi_3 \mapsto 0, \; \phi_5 \mapsto b_5.$$
In particular, to show that $\{\phi_1, \phi_2, \phi_3, \psi_3, \phi_5\}$ is algebraically independent it is sufficient to prove that $\{g_6/b_5, g_7/b_5, b_8/b_5, b_5\}$ is algebraically independent. By direct computation we find that the Jacobian $$J(g_6 / b_5, g_7/b_5, b_8 / b_5, b_5) = \frac{1}{b_5^3} J(g_6, g_7, b_8, b_5)$$ is not identically zero. By the argument we have used in the previous sections, the forms $\phi_1, \phi_2, \phi_3, \psi_3, \phi_5$ have Jacobian $\Phi_{18}$ and freely generate the ring $\mathcal{M}_*^!$.
\end{proof}

\subsection{Discriminant $-24$}

In this section $L$ is the lattice with Gram matrix $\begin{psmallmatrix} 0 & 0 & 0 & 0 & 0 & 1 \\ 0 & 0 & 0 & 0 & 1 & 0 \\ 0 & 0 & 2 & 0 & 0 & 0 \\ 0 & 0 & 0 & 12 & 0 & 0 \\ 0 & 1 & 0 & 0 & 0 & 0 \\ 1 & 0 & 0 & 0 & 0 & 0 \end{psmallmatrix}$. We need the nearly-holomorphic modular forms of weight $(-1)$ whose Fourier expansions begin \begin{align*} A_{-1}(\tau) &= q^{-1/4} e_{(0,0,1/2,0,0,0)} + q^{-1/6} (e_{(0,0,0,1/6,0,0)} + e_{(0,0,0,5/6,0,0)}) \\ &+ 2q^{-1/24} (e_{(0,0,0,1/12,0,0)} + e_{(0,0,0,5/12,0,0)} + e_{(0,0,0,7/12,0,0)} + e_{(0,0,0,11/12,0,0)}) \\ &+ 4 e_{(0,0,0,0,0,0)} \pm ... \end{align*} \begin{align*} B_{-1}(\tau) &= q^{-3/8} (e_{(0,0,0,1/4,0,0)} + e_{(0,0,0,3/4,0,0)}) - 2q^{-1/6}(e_{(0,0,0,1/6,0,0)} + e_{(0,0,0,5/6,0,0)}) \\ &- q^{-1/24} (e_{(0,0,0,1/12,0,0)} + e_{(0,0,0,5/12,0,0)} + e_{(0,0,0,7/12,0,0)} + e_{(0,0,0,11/12,0,0)}) \\ &+ 4 e_{(0,0,0,0,0,0)} \pm ... \end{align*}
\begin{align*} C_{-1}(\tau) &= q^{-1/2} e_{(0,0,0,1/2,0,0)} + q^{-1/6} (e_{(0,0,0,1/6,0,0)} + e_{(0,0,0,5/6,0,0)}) \\ &- 4q^{-1/24}  (e_{(0,0,0,1/12,0,0)} + e_{(0,0,0,5/12,0,0)} + e_{(0,0,0,7/12,0,0)} + e_{(0,0,0,11/12,0,0)}) \\ &+ 6e_{(0,0,0,0,0,0)} + ... \end{align*}
\begin{align*} D_{-1}(\tau) &= q^{-7/24} (e_{(0,0,1/2,1/12,0,0)} + e_{(0,0,1/2,11/12,0,0)}) + q^{-1/6}(e_{(0,0,0,1/6,0,0)} + e_{(0,0,0,5/6,0,0)}) \\ &+ 6q^{-1/24}(e_{(0,0,0,1/12,0,0)} + e_{(0,0,0,11/12,0,0)}) - 4q^{-1/24} (e_{(0,0,0,5/12,0,0)} + e_{(0,0,0,7/12,0,0)}) \\ &+ 6e_{(0,0,0,0,0,0)} + ... \end{align*}
\begin{align*} J_{-1}(\tau) &= q^{-1} e_{(0,0,0,0,0,0)} + 2q^{-1/6} (e_{(0,0,0,1/6,0,0)} + e_{(0,0,0,5/6,0,0)}) \\ &+ 16q^{-1/24}  (e_{(0,0,0,1/12,0,0)} + e_{(0,0,0,5/12,0,0)} + e_{(0,0,0,7/12,0,0)} + e_{(0,0,0,11/12,0,0)}) \\ &+ 34e_{(0,0,0,0,0,0)} + ... \end{align*}

These lift to Borcherds products labelled $\phi_2, \psi_2, \phi_3, \psi_3, \Phi_{17}$ respectively. The product $\phi_2$ is holomorphic but has a quadratic character under $\mathrm{SU}_{2,2}(\mathbb{Z}[\sqrt{-6}])$; the products $\psi_2, \phi_3, \psi_3$ are meromorphic with trivial character; and $\Phi_{17}$ is a skew-symmetric cusp form. In the notation of Appendix B of \cite{bw3} $$\phi_2 = \psi_2, \; \psi_2 = \frac{\psi_6^{(2)}}{\psi_2^2}, \; \phi_3 = \frac{\psi_5^{(3)}}{\psi_2}, \; \psi_3 = \frac{\psi_5^{(2)}}{\psi_2}.$$ (The form $\Phi_{17}$ does not appear in those tables as its weight is too high.) We will also need a singular additive lift of weight one whose input form is \begin{align*} F_0(\tau) &= q^{-1/6} (e_{(0,0,0,1/6,0,0)} - e_{(0,0,0,5/6,0,0)}) \\ &- 2q^{-1/24} (e_{(0,0,0,1/12,0,0)} - e_{(0,0,0,5/12,0,0)} + e_{(0,0,0,7/12,0,0)}  - e_{(0,0,0,11/12,0,0)}) \\ &\pm ...\end{align*}

The ring $\mathcal{M}_*^!$ will consist of meromorphic modular forms that are holomorphic away from $$D(1/6) = H(1/6) + H(1/24, (0,0,0,1/12,0,0)) + H(1/24, (0,0,0,5/12,1,0)).$$

\begin{theorem} $$\mathcal{M}_*^! = \mathbb{C}[\phi_1, \psi_2, \phi_3, \psi_3, \phi_2^2].$$
\end{theorem}

\begin{proof} If $\lambda := (0,0,1/2,1/12,0,0) \in L'$, then $\lambda^{\perp}$ is the root lattice $2U \oplus A_1(7)$ and its modular forms are paramodular forms of level $7$. Generators of such paramodular forms were determined in \cite{bw1}; among them are a unique (up to scalar) Gritsenko lift $g_5$ in weight $5$ and Borcherds products $b_4, b_6, b_7$ of weights $4,6,7$. Under the pullback to $\lambda^{\perp}$ the generators map as follows:
$$\phi_1 \mapsto g_5 / b_4,\; \psi_2 \mapsto b_6 / b_4, \; \phi_3 \mapsto b_7 / b_4, \; \psi_3 \mapsto 0, \; \phi_2^2 \mapsto b_4.$$ Therefore the algebraic independence of $\phi_1, \psi_2, \phi_3, \psi_3, \phi_2^2$ follows from the nonvanishing of the Jacobian $$J_0 = J(g_5 / b_4, b_6 / b_4, b_7 / b_4, b_4) = \frac{1}{b_4^3} J(g_5, b_6, b_7, b_4).$$ By direct computation we find that the fourth Fourier-Jacobi coefficient of $J_0$ is nonzero. As in the previous sections, we find that $\phi_1, \psi_2, \phi_3, \psi_3, \phi_2^2$ have Jacobian $J = \Phi_{17}$ and therefore freely generate the ring $\mathcal{M}_*^!$.
\end{proof}

\bigskip

\noindent
\textbf{Acknowledgements} 
H. Wang is grateful to Max Planck Institute for Mathematics in Bonn for its hospitality and financial support. 

\bibliographystyle{plainnat}
\bibliofont
\bibliography{hermitian_poles}

\end{document}